\documentclass[a4paper,UKenglish,cleveref,autoref]{lipics-v2021}

\title{Treewidth is Polynomial in Maximum Degree on Weakly Sparse Graphs Excluding a Planar Induced Minor}
\titlerunning{Treewidth is in $\Delta^{O(1)}$ on Weakly Sparse Graphs Excluding a Planar Induced Minor}

\author{\'{E}douard Bonnet}{Univ Lyon, CNRS, ENS de Lyon, Université Claude Bernard Lyon 1, LIP UMR5668, France \and \url{http://perso.ens-lyon.fr/edouard.bonnet/}}{edouard.bonnet@ens-lyon.fr}{https://orcid.org/0000-0002-1653-5822}{was supported by the ANR projects TWIN-WIDTH (ANR-21-CE48-0014) and Digraphs (ANR-19-CE48-0013).}
\author{Jędrzej Hodor}{Theoretical Computer Science Department, Faculty of Mathematics and Computer Science and Doctoral School of Exact and Natural Sciences, Jagiellonian University, Kraków, Poland  \and \url{https://sites.google.com/view/jedrzej-hodor}}{jedrzej.hodor@gmail.com}{https://orcid.org/0000-0002-2564-7121}{was partially supported by a Polish Ministry of Education and Science grant (Perły Nauki; PN/01/0265/2022).}
\author{Tuukka Korhonen}{Department of Informatics, University of Bergen, Norway \and \url{https://tuukkakorhonen.com/}}{tuukka.korhonen@uib.no}{https://orcid.org/0000-0003-0861-6515}{was supported by the Research Council of Norway via the project BWCA (grant no.~314528).}
\author{Tomáš Masařík}{Institute of Informatics, Faculty of Mathematics, Informatics and Mechanics, University of Warsaw, Warszawa, Poland \and \url{http://tarken.krakonos.org/research.php}}{masarik@mimuw.edu.pl}{https://orcid.org/0000-0001-8524-4036}{was supported by Polish National Science Centre SONATA-17 grant number 2021/43/D/ST6/03312.}
\authorrunning{\'E. Bonnet, J. Hodor, T. Korhonen, T. Masařík}
\Copyright{Édouard Bonnet, Jędrzej Hodor, Tuukka Korhonen, Tomáš Masařík}

\category{}

\relatedversion{The first arXiv version of this paper claimed stronger results: it did not require the absence of some biclique as a~subgraph.
  However this claim was relying on Lemma~10, a~reformulation of Lemma 4.8. in reference [20] (numbering of the first version).
  James Davies (we thank him) brought to our attention a gap in Lemma 10, and we then realized a similar flaw in the lemma of reference [20]. 
As things stand, the validity of Lemma 10 is an open question, upon which the stronger results hold.}

\acknowledgements{We thank Eunjung Kim and Jana Masaříková for discussions at an earlier stage of the project.
 This work was done at the 2023 Structural Graph Theory workshop held at the IMPAN Będlewo Conference Center (Poland) in September.
 This workshop was a~part of STRUG: Structural Graph Theory Bootcamp, funded by the ``Excellence initiative---research university (2020--2026)'' of the University of Warsaw.
We are grateful to both the organizers and the participants of the workshop for creating an enjoyable work environment.}

\nolinenumbers 

\hideLIPIcs  

\EventEditors{}
\EventNoEds{1}
\EventLongTitle{}
\EventShortTitle{}
\EventAcronym{}
\EventYear{}
\EventDate{}
\EventLocation{}
\EventLogo{}
\SeriesVolume{}
\ArticleNo{}

\usepackage[utf8]{inputenc}  

\usepackage[T1]{fontenc}
\usepackage{lmodern}

\usepackage[colorinlistoftodos,bordercolor=orange,backgroundcolor=orange!20,linecolor=orange,textsize=normalsize]{todonotes}

\usepackage{amsmath}  
\usepackage{amssymb}
\usepackage{amsthm}
\usepackage{bbm}
\usepackage{accents}
\usepackage{complexity}

\usepackage{booktabs}
\usepackage{paralist}
\usepackage{fixmath}

\makeatletter
\newtheorem*{rep@theorem}{\rep@title}
\newcommand{\newreptheorem}[2]{%
\newenvironment{rep#1}[1]{%
 \def\rep@title{#2 \ref{##1}}%
 \begin{rep@theorem}}%
 {\end{rep@theorem}}}
\makeatother

\newreptheorem{theorem}{Theorem}
\newreptheorem{lemma}{Lemma}

\usepackage{xspace}
\usepackage{tikz}

\usepackage[ruled,vlined,linesnumbered]{algorithm2e}

\usetikzlibrary{fit}
\usetikzlibrary{arrows}
\usetikzlibrary{patterns}
\usetikzlibrary{calc}
\usetikzlibrary{shapes}
\usetikzlibrary{positioning}
\usetikzlibrary{math}
\usetikzlibrary{shapes.symbols}
\usetikzlibrary{decorations.pathreplacing,calligraphy}
\usepackage[scr=boondox,scrscaled=1.05]{mathalfa}

\newcommand{\mis}{\textsc{Max Independent Set}\xspace}

\renewcommand{\geq}{\geqslant}
\renewcommand{\leq}{\leqslant}

\newcommand{\card}[1]{|{#1}|}

\newtheorem{question}{Question}

\renewcommand{\P}{{\mathcal P}}

\newcommand\tw{\text{tw}}

\newcommand\tpw{\text{tpw}}

\begin{document}

\maketitle

\begin{abstract}
  A graph $G$ contains a graph $H$ as an induced minor if $H$ can be obtained from~$G$ after vertex deletions and edge contractions.
  We show that for every $k$-vertex planar graph $H$, every graph~$G$ excluding~$H$ as an induced minor and $K_{t,t}$ as a subgraph has treewidth at~most $\Delta(G)^{f(k,t)}$ where $\Delta(G)$ denotes the maximum degree of~$G$.
  Without requiring the absence of a~$K_{t,t}$ subgraph, Korhonen [JCTB '23] has shown the upper bound of~$k^{O(1)} 2^{\Delta(G)^5}$ whose dependence in $\Delta(G)$ is exponential.  

  Our result partially answers a question of Chudnovsky [Dagstuhl seminar '23] asking whether the treewidth of graphs with $\Delta(G)=O(\log{|V(G)|})$ excluding both \mbox{a~$k$-vertex} planar graph as an induced minor and the biclique $K_{t,t}$ as a subgraph is in $O_{k,t}(\log |V(G)|)$.
  We confirm that the treewidth is in this case polylogarithmic in $|V(G)|$.
\end{abstract}

\section{Introduction}\label{sec:intro}

A~graph $G$ contains a~graph $H$ as a~\emph{minor} if $H$ can be obtained from $G$ by vertex deletions, edge deletions, and edge contractions.
The notion of \emph{induced minor} is defined similarly except edge deletions are disallowed.
The celebrated Grid Minor theorem~\cite{RobertsonS86,RobertsonST94} implies that graphs without large grid minors have low treewidth.
What can be said about the treewidth of graphs solely excluding grids as \emph{induced} minor?
Their treewidth can be arbitrarily large, as exemplified by cliques.
However, a~notable result by Korhonen is that their treewidth can be upperbounded by a~function of their maximum degree $\Delta(\cdot)$.

\begin{theorem}[\cite{Korhonen23}]\label{thm:korhonen}
  Every graph $G$ excluding a~fixed $k$-vertex planar graph as an induced minor has treewidth at~most $k^\gamma 2^{\Delta(G)^5}$ for some universal constant $\gamma$.
\end{theorem}

In this paper, we obtain a~polynomial dependence in $\Delta(G)$ if, further, arbitrarily large bicliques are excluded.

\begin{theorem}\label{thm:main}
  There is an $f: \mathbb N \times \mathbb N \to \mathbb N$ such that every graph $G$ without $K_{t,t}$ as a~subgraph nor fixed $k$-vertex planar graph as an induced minor has treewidth at~most $\Delta(G)^{f(t,k)}$.
\end{theorem}

We actually prove the following stronger statement.
\begin{theorem}\label{thm:main-stronger}
  There is an $f: \mathbb N \times \mathbb N \to \mathbb N$ such that every graph $G$ without $K_{t,t}$ as a~subgraph, and excluding as induced minors a~$k$-vertex planar graph and an~$\ell$-vertex graph has treewidth at~most $k^{O(1)} \Delta(G)^{f(t,\ell)}$.
\end{theorem}

Our tools combine well with classes of graphs that admit a~product structure; see~\cref{sec:product-structure} for the definition of the strong product $\boxtimes$ of two graphs.
More precisely, we prove the following.

\begin{theorem}\label{thm:product_structure}
  Let $H$ be a~graph of treewidth at most~$t$, and $P$ be a~path.
  Let $G$ be a~subgraph of $H \boxtimes P$ excluding a~$k$-vertex planar graph as an induced minor.
  Then the treewidth of $G$ is at~most $k^{O(1)} \cdot t^{O(1)} \cdot \Delta(G)^{O(1)}$. 
\end{theorem}


A~dependence in $\Delta(G)$ is necessary.
There are subgraphs of the strong product of a~path with a~star (hence a~graph $H$ of treewidth~1) avoiding a~planar induced minor, but whose treewidth is a~growing function of the number of vertices.
Take the~$n \times n$ grid, remove the ``vertical'' edges, and add in each ``column'' a~vertex adjacent to every vertex in the column; see~\cref{fig:pohoata-grid}.
This construction found by Pohoata~\cite{Pohoata14}, and rediscovered by~Davies~\cite{Davies22}, has treewidth $\Theta(n)$ but avoids the $5 \times 5$ grid as an induced minor.
The figure is a~proof-by-picture that these graphs are indeed subgraphs of strong products of a~path and a~star.

\begin{figure}[h!]
  \centering
  \begin{tikzpicture}[vertex/.style={draw,circle,inner sep=0.03cm}]
    \def\k{6}
    \pgfmathtruncatemacro\km{\k - 1}
    \def\sv{0.5}
    \def\sh{0.8}
    
    \foreach \i in {1,...,\k}{
      \foreach \j in {1,...,\k}{
        \node[vertex] (a\i\j) at (\i * \sh, \j *\sv) {};
      }
    }
    \foreach \i in {1,...,\k}{
      \node[vertex] (s\i) at (\i * \sh, 0) {};
    }

    \foreach \i [count = \ip from 2] in {1,...,\km}{
      \foreach \j in {1,...,\k}{
        \draw (a\i\j) -- (a\ip\j) ;
      }
    }

    \foreach \i in {1,...,\k}{
      \foreach \j in {1,...,\k}{
        \draw (s\i) to [bend left = 20] (a\i\j) ;
      }
    }
\end{tikzpicture}
\caption{The Pohoata--Davies $6 \times 6$ grid.}
\label{fig:pohoata-grid}
\end{figure}
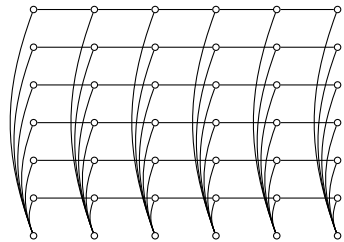

Chudnovsky~\cite[Open problem 4.1]{Chudnovsky22} asks if, when $\Delta(G) = O(\log |V(G)|)$, the treewidth of graphs excluding the $k \times k$ grid as an induced minor and the biclique $K_{t,t}$ as a~subgraph is $O_{t,k}(\log |V(G)|)$.
Our results give a~first answer to this question: The treewidth of these graphs is at~most polylogarithmic.

At first sight, Chudnovsky's question centered around forbidden induced subgraphs may look somewhat different from the setting of~\cref{thm:main}.
The two statements match since forbidding large cliques and bicliques as induced subgraphs is, by Ramsey's theorem~\cite{Ramsey30}, equivalent to excluding large bicliques as subgraphs, and forbidding a~subdivision of a~large wall or the line graph of a~subdivision of a~large wall as an induced subgraph is the same as excluding a~large grid as an induced minor.
Another simplifying feature of working with induced minors rather than induced subgraphs is that excluding as induced minor a~large grid, or a~large wall, or a~planar graph of large treewidth are all equivalent.

The motivation behind the $\Delta(G) = O(\log |V(G)|)$ condition in Chudnovsky's question is that the treewidth could in principle be logarithmic in $|V(G)|$ as well.
This would yield polynomial-time algorithms for several problems including \mis.
We come slightly short of proving it, but~\cref{thm:main} does imply a~quasipolynomial-time algorithm for \mis (and several other problems) on these graphs.

It is possible (and believed) that graphs $G$ excluding a~$k$-vertex planar graph as an induced minor have treewidth $g(k) \Delta(G)$, for some function~$g$, even without requiring the absence of $K_{t,t}$ subgraph.
This also is motivated by fast algorithms for \mis, as it would imply a~subexponential-time algorithm running in $2^{\Tilde{O}_k(\sqrt{|V(G)|})}$.
Dallard, Milanič, and Štorgel~\cite{Dallard-1} even ask whether a~(quasi)polynomial-time algorithm always exists in the absence of a~fixed planar induced minor. 
After Korhonen~\cite{Korhonen23} gave the first (very slightly) subexponential algorithm, Korhonen and Lokshtanov~\cite{KorhonenL23} provided an algorithm running in time $2^{\Tilde{O}_k(|V(G)|^{2/3})}$, which extends to the case when the forbidden induced minor is non-planar. 
There have been several recent developments in (quasi)polynomial algorithms for \mis on graphs excluding a~planar induced minor~\cite{Bonamy23,Bonnet23,Dallard-3,Dallard-2,Gartland20,Gartland23,Gartland21,PilipczukPR21}, some phrased in terms of forbidden induced subgraphs instead.

The most motivating next step would be to show \cref{thm:main} without requiring our graphs to exclude a~fixed biclique $K_{t,t}$ as a~subgraph.
Let us explicitly mention the potential further improvements by increasing difficulty.

\begin{question}\label{q:main}
  Does every graph $G$ excluding a~fixed $k$-vertex planar graph as an induced minor have, for some function $f$, treewidth at most $\Delta(G)^{f(k)}$? treewidth at most $f(k) \Delta(G)^{k^{O(1)}}$? treewidth at~most $f(k) \Delta(G)^{O(1)}$? treewidth at~most $f(k) \Delta(G)$?  
\end{question}

We note that Gartland and Lokshtanov~\cite{gl-private} conjecture the following, which would in particular imply a~positive answer to every case of the above question.

\begin{conjecture}[Gartland--Lokshtanov]
  There is a~function $f: \mathbb N \to \mathbb N$ such that every graph excluding a~fixed $k$-vertex planar graph as an induced minor has a~balanced separator dominated by at~most $f(k)$ vertices.
\end{conjecture}

Fully spelled out, the conjecture says that for every $G$ excluding a~$k$-vertex planar graph as an induced minor, there is a~set $D \subseteq V(G)$ of size at~most $f(k)$ such that $G-N[D]$ has no connected component of size larger than $|V(G)|/2$.
In particular, these graphs would have balanced separators of size $f(k)(\Delta(G)+1)$, known to imply treewidth $O(f(k) \Delta(G))$~\cite{DvorakN19}.
If true, by a~simple win-win argument, \mis could be solved in time $2^{\Tilde{O}_k(\sqrt{n})}$ on $n$-vertex graphs excluding a~$k$-vertex planar graph as an induced minor.

\section{Preliminaries}

If $i \leqslant j$ are two integers, we denote by $[i,j]$ the set of integers $\{i,i+1,\ldots,j-1,j\}$, and by~$[i]$, the set $[1,i]$.
We denote by $V(G)$ and $E(G)$ the set of vertices and edges of a graph $G$, respectively.
For $S \subseteq V(G)$, the \emph{subgraph of $G$ induced by $S$}, denoted $G[S]$, is obtained by removing from $G$ all the vertices that are not in $S$ (together with their incident edges).
Then $G-S$ is a short-hand for $G[V(G)\setminus S]$.
A~\emph{star} is a~tree with at~most one non-leaf vertex.

We denote by $N_G(v)$ and $N_G[v]$, the open, respectively closed, neighborhood of $v$ in $G$.
For $S \subseteq V(G)$, we set $N_G(S) := (\bigcup_{v \in S}N_G(v)) \setminus S$ and $N_G[S] := N_G(S) \cup S$.
We may omit the subscript if $G$ is clear from the context.

We denote by $\Delta(G)$ the maximum degree of a~graph $G$, and by $\tw(G)$, its treewidth.
A~coloring of $G$ is a~mapping $c: V(G) \to [k]$ for some natural~$k$.
It is \emph{proper} if $c(u) \neq c(v)$ holds for every $uv \in E(G)$.
We may call $c$ a~\emph{$k$-coloring}.
The sets $c^{-1}(1), \ldots, c^{-1}(k)$ are then called \emph{color classes}, with $c^{-1}(i) = \{v \in V(G)~:~c(v)=i\}$ for each $i \in [k]$.
A~\emph{star coloring} of $G$ is a proper coloring such every two color classes induce a~\emph{star forest}, i.e., a~disjoint union of stars.
The \emph{star chromatic number} (resp.~\emph{chromatic number}) of $G$ is the minimum $k$ such that $G$ admits a~star coloring (resp.~proper coloring) with $k$ color classes.   

The \emph{radius $\text{rad}(G)$} of a graph $G$ is defined as $\min_{u \in V(G)} \max_{v \in V(G)} d_G(u,v)$, where $d_G(u,v)$ is the number of edges in a shortest path between $u$ and $v$.
The \emph{radius $\text{rad}_G(S)$} of a subset of vertices $S \subseteq V(G)$ is simply defined as $\text{rad}(G[S])$.
Note that two vertices can be further away in $G[S]$ than in $G$.
A~\emph{depth-$r$ minor} $H$ of $G$, denoted by $H \preccurlyeq_r G$, is a minor of~$G$ with branch sets $B_1, \ldots, B_{\card{V(H)}}$ satisfying $\text{rad}_G(B_i) \leqslant r$ for every $i \in [\card{V(H)}]$.
In particular depth-0 minors correspond to subgraphs.
The theory of graph sparsity pioneered by Nešetřil and Ossona de Mendez~\cite{sparsity} introduces the following invariants for a graph $G$ and a class $\mathcal C$:
$$\nabla_r(G) := \underset{H \preccurlyeq_r G}{\sup}~\frac{|E(H)|}{|V(H)|},~\text{and}~\nabla_r(\mathcal C) := \underset{G \in \mathcal C}{\sup}~\nabla_r(G).$$

A class $\mathcal C$ of graphs is said to have \emph{bounded expansion} if $\nabla_r(\mathcal C) < \infty$ for every $r \in \mathbb N$.
We say that a graph \emph{$G$ has expansion $f$}, or that \emph{$f$ bounds the expansion of $G$}, if $\nabla_r(G) \leqslant f(r)$ for every $r \in \mathbb N$.

\section{Contraction--uncontraction technique}\label{sec:contraction-uncontraction}

We will need a~\emph{treewidth sparsifier}, i.e., the extraction of a~subcubic subgraph of large treewidth in a~graph of larger treewidth.
We could here use the Grid Minor theorem~\cite{RobertsonST94}, but the following result of Chekuri and Chuzhoy provides a~better lower bound in the resulting treewidth.

\begin{theorem}[\cite{Chekuri15}]\label{thm:degree-3-sparsifier}
  There is a~constant $\delta > 0$ such that every graph of treewidth $k$ admits a~subcubic subgraph of treewidth at least~$k/\log^\delta k$.
\end{theorem}

The next lemma abstracts out the contraction--uncontraction technique of the third author which, in~\cite{Korhonen23}, is specifically used over radius-2 balls.

\begin{lemma}\label{lem:sparsification-bounded-comp}
  Let $p$ be a~positive integer, $G$ be a~graph, and $F \subseteq E(G)$ be such that every connected component of the graph $(V(G),F)$ has at most $p$~vertices. 
  Then, $G$ admits an induced subgraph~$G'$ such that 
  \begin{compactitem}
  \item in $G'$, every vertex is incident to at~most three edges of $F \cap E(G')$, and
  \item $\tw(G') \geqslant \tw(G)/(p \log^\delta \tw(G))$, with $\delta$ the constant of~\cref{thm:degree-3-sparsifier}.
  \end{compactitem}
\end{lemma}

\begin{proof}
  Let $\P$ be the partition $\{P_1, \ldots, P_h\}$ of $V(G)$ into the vertex sets of the connected components of~$(V(G),F)$.
  It follows that $|P_i| \leqslant p$ for every $i \in [h]$.
  In particular, $\tw(G/\P) \geqslant \tw(G)/p$.
  Indeed, a~tree-decomposition of~$G/\P$ of width at most $\tw(G)/p-1$ could be turned into a~tree-decomposition of~$G$ of width at most~$\tw(G)/p \cdot \max_{i \in [h]} |P_i| - 1 \leqslant \tw(G)-1$, simply by flattening the parts of $\P$ in each bag, leading to a~contradiction.
  On the other hand, $\tw(G/\P) \leqslant \tw(G)$ since $G/\P$ is obtained from~$G$ by performing edge contractions, as each $P_i$ is connected. 

  By \cref{thm:degree-3-sparsifier} applied to $G/\P$, there is a~subcubic subgraph $H$ of $G/\P$ with 
  $$\tw(H) \geq \frac{\tw(G/\P)}{\log^\delta \tw(G/\P)} \geqslant \frac{\tw(G)}{p \log^\delta \tw(G)}.$$

  We now build an induced subgraph $G'$ of $G$ having $H$ as a~minor (hence at least its treewidth) such that every vertex of $G'$ is incident to at~most three edges of~$F$.
  As $H$ is subcubic, each $P \in V(H)$ is incident to at most three edges of $H$.
  From each $P \in V(H)$, let us keep a~minimal subset $P' \subseteq P$ such that $G[P']$ is connected and $G[\bigcup_{P \in V(H)} P']/\P'$ still contains $H$ as a~subgraph, where $\P' := \{P' : P \in V(H)\}$.
  By \emph{minimal} we mean that for each $P' \in V(H)$, the removal of any vertex in $P'$ breaks one of the latter conditions.
  
  Note that each $P' \in V(H)$ comprises up to three \emph{terminals} realizing the up-to-three edges in $H$, plus a~minimal subset connecting these three terminals in $P$.
  Therefore, if $P'$ would contain a~vertex $v$ with more than three neighbors in $P'$, we could delete one of its neighbors by taking shortest paths from $v$ to the terminals in $G[P']$ and deleting a~neighbor not used in these shortest paths.
  This implies that every vertex of $P'$ is incident to at~most three edges of~$F$ in $G[P']$, since no edge of $F$ can have exactly one endpoint in~$P$.

  Thus we set $G' := G[\bigcup_{P' \in V(H)} P']$, and get $\tw(G') \geqslant \tw(H) \geqslant \tw(G)/(p \log^\delta \tw(G))$.
\end{proof}

\section{Star coloring with constantly many colors}\label{sec:star-coloring}

Building on a classic result by Kühn and Osthus~\cite{Kuhn04}, Dvořák showed the following.

\begin{theorem}[\cite{Dvorak18}]\label{thm:dvorak}
  For every non-negative integer $t$ and graph $H$, there is a function $f_{t,H}: \mathbb N \to \mathbb N$ such that every graph without $K_{t,t}$ subgraph nor induced subdivision of $H$ has expansion $f_{t,H}$.
\end{theorem}

Kühn and Osthus showed the same statement with the weaker conclusion that the degeneracy is bounded by a function of $t$ and~$H$.
In turn, by the work of Nešetřil and Ossona de Mendez, graphs of bounded expansion have bounded star chromatic number.

\begin{theorem}[\cite{Nesetril08}, Theorem 7.1, (5) $\Rightarrow$ (3) with $p=2$]\label{thm:star-coloring}
  Every graph class with bounded expansion has bounded star chromatic number.
\end{theorem}

We also observe the following.
\begin{observation}\label{obs:ind-subd-minor}
  Every graph excluding a~graph $H$ as an induced minor also excludes $H$ as an induced subdivision.
\end{observation}

Combining~\cref{thm:dvorak,thm:star-coloring,obs:ind-subd-minor} we get a~bounded star coloring for our graphs of interest.
\begin{theorem}\label{thm:bd-star-coloring}
  There is a~function $c: \mathbb N \times \mathbb N \to \mathbb N$ such that every graph without $K_{t,t}$ as a~subgraph nor fixed $k$-vertex graph $H$ as an induced minor admits a~star $c(t,k)$-coloring. 
\end{theorem}
Note that $H$ need not be planar in~\cref{thm:bd-star-coloring}.

\section{Reduced number of sparsification rounds}

We now use the contraction--uncontraction technique~\cite{Korhonen23}; see~\cref{sec:contraction-uncontraction}.
In a~first step, we lower the maximum degree.
In a~second step, we simply use Korhonen's result (see~\cref{thm:korhonen}) on an induced subgraph of low maximum degree.
The crucial difference with~\cite{Korhonen23} is that the number of rounds does not depend on the ``initial'' maximum degree $\Delta(G)$ but solely on $t$ and $k$ (such that $K_{t,t}$ is not a~subgraph of~$G$, and $G$ excludes a~$k$-vertex planar graph).

We successively apply \cref{lem:sparsification-bounded-comp} ${c(t,k) \choose 2}$ times, where $c(t,k)$ is the function of~\cref{thm:bd-star-coloring}, on stars formed by every pair of color classes in a~star coloring.

\begin{lemma}\label{lem:weakly-sparse-to-sparse}
  Let $t$ be a positive integer, and $H$ be a~fixed $k$-vertex graph. 
  Every graph $G$ without $K_{t,t}$ as a~subgraph nor $H$ as an induced minor has an induced subgraph $\widehat{G}$ such that
  \begin{compactitem}
    \item $\Delta(\widehat G) \leqslant 3(c-1)$, and 
    \item $\tw(\widehat G) \geqslant \tw(G)/((\Delta(G)+1) \log^\delta \tw(G))^{c \choose 2}$,
  \end{compactitem}
  where $c := c(t,k)$ is as in~\cref{thm:bd-star-coloring}.
\end{lemma}
\begin{proof}
  Let $A_1, \ldots, A_c$ be the color classes of a~star coloring of~$G$ given by~\cref{thm:bd-star-coloring}.
  For every unordered pair $i \neq j \in [c]$, $G[A_i \cup A_j]$ is a~star forest, a~property that is closed under taking induced subgraphs.
  We set $F_{ij} := E(G[A_i \cup A_j]) \subseteq E(G)$ and $q := {c \choose 2}$.

  We build a~chain for the \emph{induced subgraph} relation $G=G_0, G_1 \subseteq_i G_2 \subseteq_i  \ldots \subseteq_i G_{q - 1} \subseteq_i G_q=\widehat G$, in the following way.
  We (bijectively) list the unordered pairs $i \neq j \in [c]$ from 1 to $q$.
  We obtain $G_h$, where $h$ corresponds to the pair $\{i,j\}$, by applying \cref{lem:sparsification-bounded-comp} on the triple $G_{h-1}$, $F := F_{ij} \cap E(G_{h-1})$, and $p := \Delta(G)+1$.
  We recall that \cref{lem:sparsification-bounded-comp} takes in addition to a~graph (here $G_{h-1}$), an edge subset $F$, and an integer $p$.
  As $(V(G),F_{ij})$ is a~star forest, so is its induced subgraph $(V(G_{h-1}), F)$.
  Thus we indeed have that every connected component $(V(G_{h-1}), F)$ has at~most~$p = \Delta(G)+1$ vertices.

  We get that $\tw(G_p) \geqslant \tw(G_{p-1})/((\Delta(G)+1) \log^\delta \tw(G))$.
  It thus eventually holds that $\tw(\widehat G) = \tw(G_q) \geqslant \tw(G)/((\Delta(G)+1) \log^\delta \tw(G))^q$.
  Fix any $i \in [c]$ and $v \in V(\widehat G) \cap A_i$.
  For every $j \in [c] \setminus \{i\}$, at~most three edges of $F_{ij}$ can be incident to $v$ in $\widehat G$.
  So $v$ has degree at~most $3(c-1)$ in $\widehat G$.
  Thus $\widehat G$ satisfies the claimed properties.
\end{proof}

We can now prove our main theorem, whose statement we recall for convenience.

\begin{reptheorem}{thm:main}\label{thm:main-recalled}
  There is an $f: \mathbb N \times \mathbb N \to \mathbb N$ such that every graph $G$ without $K_{t,t}$ as a~subgraph nor fixed $k$-vertex planar graph as an induced minor has treewidth at~most $\Delta(G)^{f(t,k)}$.
\end{reptheorem}
\begin{proof}
  Let $c := c(t,k)$ be as in~\cref{thm:bd-star-coloring}, $q := {c \choose 2}$, $\delta$ be the constant of~\cref{thm:degree-3-sparsifier}, $\gamma$ be that of~\cref{thm:korhonen}, and $g(t,k)$ be the largest integer $s$ such that $\log^{\delta q} s > \sqrt s$.
  We can assume that $\Delta(G) \geqslant 3$ since otherwise the validity of the theorem statement is clear.
  By~\cref{lem:weakly-sparse-to-sparse}, $G$ admits an induced subgraph $\widehat G$ with maximum degree at~most $3(c-1)$ and treewidth at~least $\tw(G)/((\Delta(G)+1) \log^\delta \tw(G))^q$.

  As $\widehat G$ satisfies the same hereditary properties as~$G$, by~\cref{thm:korhonen}, its treewidth is at~most $k^\gamma 2^{(3(c-1))^5}$.
  Therefore, $$\tw(G)/((\Delta(G)+1) \log^\delta \tw(G))^q \leqslant \tw(\widehat G) \leqslant k^\gamma 2^{(3(c-1))^5},~\text{thus}$$
  $$\tw(G)/\log^{\delta q}\tw(G) \leqslant k^\gamma 2^{(3(c-1))^5} \cdot (\Delta(G)+1)^q.$$
  Either $\tw(G) \leqslant g(t,k)$ (and we are done as long as $f(t,k) \geqslant g(t,k)$) or $\tw(G)/\log^{\delta q}\tw(G) \geqslant \sqrt{tw(G)}$.
  In the latter case,
  $$\tw(G) \leqslant k^{2 \gamma} 2^{2(3(c-1))^5} \cdot (\Delta(G)+1)^{2q} \leqslant k^{2 \gamma} 2^{2(3(c-1))^5} 2^{2q} \cdot \Delta(G)^{2q}.$$
  We conclude by choosing $f(t,k) := \max \left(g(t,k), \left\lceil \log \left(k^{4 \gamma} 2^{972(c(t,k)-1)^5} 2^{4{c(t,k) \choose 2}}\right) \right\rceil \right)$.
\end{proof}

As~\cref{lem:weakly-sparse-to-sparse} does not require the excluded induced minor to be planar, we proved:
\begin{reptheorem}{thm:main-stronger}\label{thm:main-stronger-recalled}
There is an $f: \mathbb N \times \mathbb N \to \mathbb N$ such that every graph $G$ without $K_{t,t}$ as a~subgraph, and excluding as induced minors a~$k$-vertex planar graph and an~$\ell$-vertex graph has treewidth at~most $k^{O(1)} \Delta(G)^{f(t,\ell)}$.
\end{reptheorem}
  
\section{Clustered edge-colorings}

The combination of \cref{sec:contraction-uncontraction,sec:star-coloring} suggests the use of (non-necessarily proper) edge-colorings any connected component induced by any monochromatic component of which has small size.  
This is referred to as \emph{clustered edge-coloring}.
More precisely, an edge-coloring of a~graph $G$ has \emph{clustering~$p$} if for every monochromatic component $F$, every connected component of $(V(G),F)$ has at~most $p$~vertices.  
For instance, an edge-coloring with clustering~2 is a~proper edge-coloring.

\begin{lemma}\label{lem:sparsification-bounded-clustered-col}
  Let $p$ be a~positive integer, $G$ be a~graph, and $F_1, \ldots, F_h$ be the color classes of an edge-coloring of $G$ with clustering~$p$.
  Then, $G$ admits an induced subgraph~$G'$ such that 
  \begin{compactitem}
  \item $\Delta(G') \leqslant 3h$, and
  \item $\tw(G') \geqslant \tw(G)/(p \log^\delta \tw(G))^h$, with $\delta$ the constant of~\cref{thm:degree-3-sparsifier}.
  \end{compactitem}
\end{lemma}

\begin{proof}
  Set $G_0 := G$.
  For every $i \in [h]$ going from $1$ to $h$, let $G_i$ be the induced subgraph of~$G_{i-1}$ obtained by applying \cref{lem:sparsification-bounded-comp} with edge subset $F := F_i \cap E(G_{i-1})$.
  We then define $G'$ as $G_h$.
  By the first item of~\cref{lem:sparsification-bounded-comp}, every vertex of $G'$ has at most three incident edges in $F_i \cap E(G')$, hence has degree at most $3h$.
  The second item readily follows from that of~\cref{lem:sparsification-bounded-comp}.
\end{proof}

We show an upper bound on the treewidth of graphs excluding a~grid as an induced minor and admitting edge-colorings with few colors and moderately large clustering.

\begin{lemma}\label{lem:edge_clustered_coloring_gives_small_tw}
   Every graph $G$ excluding a~$k$-vertex planar graph as an induced minor and admitting an $h$-edge-coloring with clustering $c > 0$ has treewidth at most $ k^{O(1)} 2^{O(h^5 + h \log c)}$.
\end{lemma}
\begin{proof}
    By~\cref{lem:sparsification-bounded-clustered-col}, $G$ admits an induced subgraph $G'$ of maximum degree at~most $3h$ and
  $$\tw(G)/(c \log^\delta \tw(G))^h \leqslant \tw(G'),$$ with $\delta$ the constant of~\cref{thm:degree-3-sparsifier}.

  As $G$ excludes a~$k$-vertex planar graph as an induced minor, so does $G'$.
  Thus by~\cref{thm:korhonen}, $$\tw(G') \leqslant k^\gamma 2^{\Delta(G')^5} \leq k^\gamma 2^{243 h^5},$$
  for some universal constant $\gamma$.

  From the two previous inequalities, we get that $$\tw(G)/(c \log^\delta \tw(G))^h \leqslant k^\gamma 2^{243 h^5}.$$
  If $\log^{\delta h} \tw(G) \leqslant \sqrt{\tw(G)}$, we get that
  $$\tw(G) \leqslant c^{2h} k^{2\gamma} 2^{2 \cdot 243 h^5} = k^{O(1)} 2^{O(h^5+h \log c)},$$
  as claimed.
  If instead $\log^{\delta h} \tw(G) > \sqrt{\tw(G)}$, the statement of the lemma also holds, as then $\tw(G) = 2^{O(h^2)}$.
\end{proof}

\section{Clusters of bounded treewidth}\label{sec:bd-tw-clustering}

The goal of this section is to relax the notion of \emph{clustering} of edge-colorings so that \cref{lem:edge_clustered_coloring_gives_small_tw} still holds.
Namely, we now allow clusters to be arbitrarily large, however, we want their treewidth to be bounded.
This can be converted into an edge-coloring (still with few colors) with bounded clustering.
Indeed, we show that a~graph $G$ of bounded treewidth admits a~$3$-edge-coloring with clustering $f(\Delta(G))$.

The main tool we plan to use is the notion of \emph{tree-partitions} of graphs.
A~pair $(T,\{B_x : x \in V(T)\})$ is a~tree-partition of a~graph $G$ if $T$ is a~tree, and $\{B_x : x \in V(T)\}$ is a~partition of $V(G)$ such that for every $uv \in E(G)$, there exist a~pair $x,y \in V(T)$ of equal or adjacent vertices such that $u \in B_x,v \in B_y$.
The \emph{width} of a~tree-partition $(T,\{B_x : x \in V(T)\})$ is defined as the maximum cardinality of an element of $\{B_x : x \in V(T)\}$.
The \emph{tree-partition width} of a~graph $G$, denoted $\tpw(G)$, is the minimum width of a~tree-partition of $G$.
An anonymous referee of \cite{DO95} showed that every graph $G$ has tree-partition width of at most $24\tw(G)\Delta(G)$ (see also \cite{W09,DW23}).

\begin{lemma}\label{lem:clustered_edge_coloring_for_bounded_tw}
    Every graph $G$ admits a~$3$-edge-coloring with clustering $\tpw(G)(\Delta(G)+1)$, which is in particular $O(\tw(G)\Delta(G)^2)$.
\end{lemma}
\begin{proof}
    Let $(T,\{B_x : x \in V(T)\})$ be a~tree-partition of $G$ of width $w$.
    We root $T$ at an arbitrary vertex $r$.
    Assign to each vertex $x \in V(T)$ its distance to $r$ in~$T$, denoted $\text{depth}(x)$.
    We define a~coloring $\text{col} : E(G) \rightarrow \{0,1,2\}$ as follows.
    Let $uv \in E(G)$, and let $x \in V(T), y \in N_T[x]$ be such that $u \in B_x$ and $v \in B_y$.
    Without loss of generality assume that $\text{depth}(x) \leq \text{depth}(y)$.
    If $x \neq y$, then we set $\text{col}(uv) = \text{depth}(x) \text{ mod } 2$, and otherwise, we set $\text{col}(uv) = 2$.
    Every monochromatic connected component of color~$2$ is contained in a~single part $B_x$ for some $x \in V(T)$, and so, its cardinality is at most $w$.
    On the other hand, for every monochromatic connected component of color $0$ or $1$, there exists a~single part $B_x$ for some $x \in V(T)$ such that every edge in the component is incident to a~vertex in $B_x$.
    It follows that the size of this monochromatic component is at most $w \cdot (\Delta(G)+1)$.
\end{proof}

Now, we state and prove a~relaxed version of \cref{lem:edge_clustered_coloring_gives_small_tw}.

\begin{lemma}\label{lem:edge_clustered_tw_gives_small_tw}
  Suppose graph $G$ excludes as an induced minor a~$k$-vertex planar graph and admits an edge-coloring $\text{col}_1$ with color classes $F_1,\dots,F_h$ such that for each $i \in [h]$, the graph $(V(G),F_i)$ has treewidth at most $t$.
  Then the treewidth of $G$ is at most $k^{O(1)} t^{O(h^5)}\Delta(G)^{O(h^5)}$.
\end{lemma}
\begin{proof}
    By \cref{lem:clustered_edge_coloring_for_bounded_tw}, for each $i \in [h]$, the graph $(V(G),F_i)$ admits a~$3$-edge-coloring with clustering $O(t\Delta(G)^2)$.
    Since $\{F_1,\dots,F_h\}$ is a~partition of $V(G)$, the above edge-colorings give a~$3$-edge-coloring $\text{col}_2$ of $E(G)$.
    Consider the product edge-coloring $\text{col}$ of $\text{col}_1$ and $\text{col}_2$ of $E(G)$, that is, $\text{col}(e)=(\text{col}_1(e),\text{col}_2(e))$ for every $e \in E(G)$.
    Observe that $\text{col}$ uses at most $3h$ colors and has clustering $O(t\Delta(G)^2)$.
    Finally, by \cref{lem:edge_clustered_coloring_gives_small_tw}, we obtain 
    $$\tw(G) \leq k^{O(1)} \cdot (t\Delta(G)^2)^{O(h^5)} = k^{O(1)} t^{O(h^5)}\Delta(G)^{O(h^5)},$$
    as claimed.
\end{proof}

\section{Product structure}\label{sec:product-structure}

The \emph{strong product of graphs $H_1$ and $H_2$}, denoted by $H_1 \boxtimes H_2$, is the graph with vertex set $V(H_1) \times V(H_2)$ such that there is an edge $(u,v)(u',v')$ whenever either $u = u'$ and $vv' \in E(H_2)$, or $uu' \in E(H_1)$ and $v = v'$, or $uu' \in E(H_1)$ and $vv' \in E(H_2)$.
We prove the following theorem.

\begin{reptheorem}{thm:product_structure}
    Let $H$ be a~graph of treewidth at most~$t$, and $P$ be a~path.
  Let $G$ be a~subgraph of $H \boxtimes P$ excluding a~$k$-vertex planar graph as an induced minor.
  Then the treewidth of $G$ is at~most $k^{O(1)} \cdot t^{O(1)} \cdot \Delta(G)^{O(1)}$.
\end{reptheorem}
\begin{proof}
    We claim that $H \boxtimes P$ admits a~$3$-edge-coloring such that if $F$ is any of its color classes, then the graph $(V(H \boxtimes P),F)$ has treewidth at most $2t$.
    First, note that this suffices to prove the theorem.
    Indeed, we can restrict this edge-coloring to $G$ and apply \cref{lem:edge_clustered_tw_gives_small_tw} with $h=3$ to end the proof.

    Let us justify the initial claim.
    We construct a~coloring $\text{col} : E(H \boxtimes P) \rightarrow \{0,1,2\}$.
    Let $P = v_1v_2 \dots v_m$.
    We set the color of each edge $(u,v)(u',v')$ such that $v = v'$ to~$2$, and each edge $(u,v)(u',v')$ such that $v \neq v'$ to $i \text{ mod } 2$, where $i$ the~positive integer satisfying $\{v,v'\} = \{v_i,v_{i+1}\}$.
    The graph $G$ restricted to edges of color $2$ is simply a~disjoint union of copies of $H$, hence, it has treewidth at most $t$.
    On the other hand, the graph $G$ restricted to edges of color $0$ or $1$ is a~disjoint union of copies of the graph $H' = H \boxtimes K_2$.
    Thus $\tw(H') \leq 2\tw(H) \leq 2t$, which ends the proof.
\end{proof}


\begin{thebibliography}{10}

\bibitem{Bonamy23}
Marthe Bonamy, \'Edouard Bonnet, Hugues D{\'{e}}pr{\'{e}}s, Louis Esperet,
  Colin Geniet, Claire Hilaire, St{\'{e}}phan Thomass{\'{e}}, and Alexandra
  Wesolek.
\newblock Sparse graphs with bounded induced cycle packing number have
  logarithmic treewidth.
\newblock In Nikhil Bansal and Viswanath Nagarajan, editors, {\em Proceedings
  of the 2023 {ACM-SIAM} Symposium on Discrete Algorithms, {SODA} 2023,
  Florence, Italy, January 22--25, 2023}, pages 3006--3028. {SIAM}, 2023.
\newblock \href {https://doi.org/10.1137/1.9781611977554.ch116}
  {\path{doi:10.1137/1.9781611977554.ch116}}.

\bibitem{Bonnet23}
\'{E}douard Bonnet, Julien Duron, Colin Geniet, St\'{e}phan Thomass\'{e}, and
  Alexandra Wesolek.
\newblock Maximum independent set when excluding an induced minor: ${K}_1+
  t{K}_2$ and $t{C}_3 \uplus {C}_4$.
\newblock In Inge~Li G{\o}rtz, Martin Farach-Colton, Simon~J. Puglisi, and
  Grzegorz Herman, editors, {\em 31st Annual European Symposium on Algorithms
  (ESA 2023)}, volume 274 of {\em Leibniz International Proceedings in
  Informatics (LIPIcs)}, pages 23:1--23:15, Dagstuhl, Germany, 2023. Schloss
  Dagstuhl -- Leibniz-Zentrum f{\"u}r Informatik.
\newblock \href {https://doi.org/10.4230/LIPIcs.ESA.2023.23}
  {\path{doi:10.4230/LIPIcs.ESA.2023.23}}.

\bibitem{Chekuri15}
Chandra Chekuri and Julia Chuzhoy.
\newblock Degree-3 treewidth sparsifiers.
\newblock In Piotr Indyk, editor, {\em Proceedings of the Twenty-Sixth Annual
  {ACM-SIAM} Symposium on Discrete Algorithms, {SODA} 2015, San Diego, CA, USA,
  January 4--6, 2015}, pages 242--255. {SIAM}, 2015.
\newblock \href {https://doi.org/10.1137/1.9781611973730.19}
  {\path{doi:10.1137/1.9781611973730.19}}.

\bibitem{Chudnovsky22}
Maria Chudnovsky, Neeldhara Misra, Dani\"el Paulusma, Oliver Schaudt, and
  Akanksha Agrawal.
\newblock {Vertex Partitioning in Graphs: From Structure to Algorithms
  (Dagstuhl Seminar 22481)}.
\newblock {\em Dagstuhl Reports}, 12(11):109--123, 2023.
\newblock \href {https://doi.org/10.4230/DagRep.12.11.109}
  {\path{doi:10.4230/DagRep.12.11.109}}.

\bibitem{Dallard-1}
Cl{\'{e}}ment Dallard, Martin Milanič, and Kenny Štorgel.
\newblock Treewidth versus clique number. {I.} graph classes with a forbidden
  structure.
\newblock {\em {SIAM} J. Discret. Math.}, 35(4):2618--2646, 2021.
\newblock \href {https://doi.org/10.1137/20M1352119}
  {\path{doi:10.1137/20M1352119}}.

\bibitem{Dallard-3}
Cl{\'{e}}ment Dallard, Martin Milanič, and Kenny Štorgel.
\newblock Treewidth versus clique number. {III.} tree-independence number of
  graphs with a forbidden structure.
\newblock {\em CoRR}, abs/2206.15092, 2022.
\newblock \href {http://arxiv.org/abs/2206.15092} {\path{arXiv:2206.15092}}.

\bibitem{Dallard-2}
Clément Dallard, Martin Milanič, and Kenny Štorgel.
\newblock Treewidth versus clique number. {II.} tree-independence number, 2021.
\newblock \href {https://doi.org/10.48550/ARXIV.2111.04543}
  {\path{doi:10.48550/ARXIV.2111.04543}}.

\bibitem{Davies22}
James Davies.
\newblock Oberwolfach report 1/2022.
\newblock 2022.
\newblock \href {https://doi.org/10.4171/OWR/2022/1}
  {\path{doi:10.4171/OWR/2022/1}}.

\bibitem{DO95}
Guoli Ding and Bogdan Oporowski.
\newblock Some results on tree decomposition of graphs.
\newblock {\em Journal of Graph Theory}, 20(4):481--499, 1995.
\newblock \href {https://doi.org/10.1002/jgt.3190200412}
  {\path{doi:10.1002/jgt.3190200412}}.

\bibitem{DW23}
Marc Distel and David~R. Wood.
\newblock Tree-partitions with small bounded degree trees, 2023.
\newblock \href {http://arxiv.org/abs/2210.12577} {\path{arXiv:2210.12577}}.

\bibitem{DMW23}
Vida Dujmović, Pat Morin, and David~R. Wood.
\newblock Graph product structure for non-minor-closed classes.
\newblock {\em Journal of Combinatorial Theory, Series B}, 162:34--67, 2023.
\newblock \href {https://doi.org/10.1016/j.jctb.2023.03.004}
  {\path{doi:10.1016/j.jctb.2023.03.004}}.

\bibitem{Dvorak18}
Zdenek Dvor{\'{a}}k.
\newblock Induced subdivisions and bounded expansion.
\newblock {\em European Journal of Combinatorics}, 69:143--148, 2018.
\newblock \href {https://doi.org/10.1016/j.ejc.2017.10.004}
  {\path{doi:10.1016/j.ejc.2017.10.004}}.

\bibitem{DvorakN19}
Zdeněk Dvořák and Sergey Norin.
\newblock Treewidth of graphs with balanced separations.
\newblock {\em J. Comb. Theory, Ser. {B}}, 137:137--144, 2019.
\newblock \href {https://doi.org/10.1016/J.JCTB.2018.12.007}
  {\path{doi:10.1016/J.JCTB.2018.12.007}}.

\bibitem{Gartland20}
Peter Gartland and Daniel Lokshtanov.
\newblock Independent set on ${P}_k$-free graphs in quasi-polynomial time.
\newblock In Sandy Irani, editor, {\em 61st {IEEE} Annual Symposium on
  Foundations of Computer Science, {FOCS} 2020, Durham, NC, USA, November
  16--19, 2020}, pages 613--624. {IEEE}, 2020.
\newblock \href {https://doi.org/10.1109/FOCS46700.2020.00063}
  {\path{doi:10.1109/FOCS46700.2020.00063}}.

\bibitem{gl-private}
Peter Gartland and Daniel Lokshtanov.
\newblock private communication, 2023.

\bibitem{Gartland23}
Peter Gartland, Daniel Lokshtanov, Tomáš Masařík, Marcin Pilipczuk,
  Micha\l{} Pilipczuk, and Pawe\l{} Rzążewski.
\newblock Maximum weight independent set in graphs with no long claws in
  quasi-polynomial time.
\newblock {\em CoRR}, abs/2305.15738, 2023.
\newblock Accepted at {ACM} {SIGACT} Symposium on Theory of Computing (STOC)
  2024.
\newblock \href {http://arxiv.org/abs/2305.15738} {\path{arXiv:2305.15738}}.

\bibitem{Gartland21}
Peter Gartland, Daniel Lokshtanov, Marcin Pilipczuk, Micha\l{} Pilipczuk, and
  Pawe\l{} Rzążewski.
\newblock Finding large induced sparse subgraphs in ${C}_{>t}$-free graphs in
  quasipolynomial time.
\newblock In Samir Khuller and Virginia~Vassilevska Williams, editors, {\em
  {STOC} '21: 53rd Annual {ACM} {SIGACT} Symposium on Theory of Computing,
  Virtual Event, Italy, June 21--25, 2021}, pages 330--341. {ACM}, 2021.
\newblock \href {https://doi.org/10.1145/3406325.3451034}
  {\path{doi:10.1145/3406325.3451034}}.

\bibitem{Korhonen23}
Tuukka Korhonen.
\newblock Grid induced minor theorem for graphs of small degree.
\newblock {\em Journal of Combinatorial Theory, Series B}, 160:206--214, 2023.
\newblock \href {https://doi.org/10.1016/j.jctb.2023.01.002}
  {\path{doi:10.1016/j.jctb.2023.01.002}}.

\bibitem{KorhonenL23}
Tuukka Korhonen and Daniel Lokshtanov.
\newblock Induced-minor-free graphs: Separator theorem, subexponential
  algorithms, and improved hardness of recognition.
\newblock {\em CoRR}, abs/2308.04795, 2023.
\newblock \href {http://arxiv.org/abs/2308.04795} {\path{arXiv:2308.04795}}.

\bibitem{Kuhn04}
Daniela K{\"{u}}hn and Deryk Osthus.
\newblock Induced subdivisions in ${K}_{s, s}$-free graphs of large average
  degree.
\newblock {\em Comb.}, 24(2):287--304, 2004.
\newblock \href {https://doi.org/10.1007/s00493-004-0017-8}
  {\path{doi:10.1007/s00493-004-0017-8}}.

\bibitem{sparsity}
Jaroslav Nesetril and Patrice~Ossona de~Mendez.
\newblock {\em Sparsity - Graphs, Structures, and Algorithms}, volume~28 of
  {\em Algorithms and combinatorics}.
\newblock Springer, 2012.
\newblock \href {https://doi.org/10.1007/978-3-642-27875-4}
  {\path{doi:10.1007/978-3-642-27875-4}}.

\bibitem{Nesetril08}
Jaroslav Nešetřil and Patrice~Ossona de~Mendez.
\newblock Grad and classes with bounded expansion i. decompositions.
\newblock {\em Eur. J. Comb.}, 29(3):760--776, 2008.
\newblock \href {https://doi.org/10.1016/j.ejc.2006.07.013}
  {\path{doi:10.1016/j.ejc.2006.07.013}}.

\bibitem{PilipczukPR21}
Marcin Pilipczuk, Micha\l{} Pilipczuk, and Pawe\l{} Rzążewski.
\newblock Quasi-polynomial-time algorithm for independent set in ${P}_t$-free
  graphs via shrinking the space of induced paths.
\newblock In Hung~Viet Le and Valerie King, editors, {\em 4th Symposium on
  Simplicity in Algorithms, {SOSA} 2021, Virtual Conference, January 11--12,
  2021}, pages 204--209. {SIAM}, 2021.
\newblock \href {https://doi.org/10.1137/1.9781611976496.23}
  {\path{doi:10.1137/1.9781611976496.23}}.

\bibitem{Pohoata14}
Andrei~Cosmin Pohoata.
\newblock {\em Unavoidable induced subgraphs of large graphs}.
\newblock Senior theses, Princeton University, 2014.
\newblock URL: \url{http://arks.princeton.edu/ark:/88435/dsp012514nk67q}.

\bibitem{Ramsey30}
Frank~P. Ramsey.
\newblock On a problem of formal logic.
\newblock In {\em Proc. London Math. Soc. series 2}, volume~30 of {\em
  264--286}, 1930.
\newblock \href {https://doi.org/10.1112/plms/s2-30.1.264}
  {\path{doi:10.1112/plms/s2-30.1.264}}.

\bibitem{RobertsonS86}
Neil Robertson and Paul~D. Seymour.
\newblock Graph minors. {V.} excluding a planar graph.
\newblock {\em Journal of Combinatorial Theory, Series B}, 41(1):92--114, 1986.
\newblock \href {https://doi.org/10.1016/0095-8956(86)90030-4}
  {\path{doi:10.1016/0095-8956(86)90030-4}}.

\bibitem{RobertsonST94}
Neil Robertson, Paul~D. Seymour, and Robin Thomas.
\newblock Quickly excluding a planar graph.
\newblock {\em Journal of Combinatorial Theory, Series B}, 62(2):323--348,
  1994.
\newblock \href {https://doi.org/10.1006/jctb.1994.1073}
  {\path{doi:10.1006/jctb.1994.1073}}.

\bibitem{W09}
David~R. Wood.
\newblock On tree-partition-width.
\newblock {\em European Journal of Combinatorics}, 30(5):1245--1253, 2009.
\newblock Part Special Issue on Metric Graph Theory.
\newblock \href {https://doi.org/10.1016/j.ejc.2008.11.010}
  {\path{doi:10.1016/j.ejc.2008.11.010}}.

\end{thebibliography}

\end{document}